\documentclass[runningheads]{llncs}
\usepackage{amssymb,amsmath,latexsym}
\usepackage{enumitem}
\usepackage{xcolor}
\usepackage[normalem]{ulem}

\newcommand{\F}{\mathbb F}
\DeclareMathOperator{\Tr}{Tr}

\DeclareMathOperator{\im}{im}

\DeclareMathOperator{\Ker}{ker}
\DeclareMathOperator{\rad}{rad}
\begin{document}
\title{On subspaces of Kloosterman zeros and permutations of the form $L_1(x^{-1})+L_2(x)$}
\titlerunning{Subspaces of Kloosterman zeros and permutation polynomials}
\author{Faruk G\"olo\u{g}lu \inst{1} \thanks{Faruk G\"olo\u{g}lu was supported by the GA\v{C}R Grant 18-19087S -301-13/201843}  \and
Lukas K\"olsch\inst{2} \and
Gohar Kyureghyan\inst{2} \and
L\'eo Perrin\inst{3}}
\authorrunning{Faruk G\"olo\u{g}lu et al.}
\institute{Department of Mathematics, Faculty of Mathematics and Physics, Charles University, Prague, Czech Republic \\
\email{farukgologlu@gmail.com}\\
\and
Department of Mathematics,  University of Rostock, Germany\\
\email{\{lukas.koelsch,gohar.kyureghyan\}@uni-rostock.de}\\
\and
Inria, Paris, France \\
\email{leo.perrin@inria.fr}}

\maketitle
\abstract{Permutations of the form $F(x)=L_1(x^{-1})+L_2(x)$  with linear functions $L_1,L_2$ are closely related to several interesting questions regarding CCZ-equivalence and EA-equivalence of the inverse function. In this paper, we show that $F$ cannot be a permutation on binary fields if the kernel of $L_1$ or $L_2$ is  large. A key step of our proof is  an observation on the maximal size of a subspace $V$ of $\F_{2^n}$ that  consists of Kloosterman zeros, i.e. a subspace $V$ such that $K_n(v)=0$ for every $v \in V$ where $K_n(v)$ denotes the Kloosterman sum of $v$.}

\section{Introduction}

Vectorial Boolean functions play an important  role in the design of symmetric cryptosystems as design choices for S-boxes. The linear and differential properties of vectorial Boolean functions are a measure of resistance against linear \cite{linattacks} and differential \cite{diffattack} attacks. 
\begin{definition}
	A function $F \colon \F_{2^n} \rightarrow \F_{2^n}$ has differential uniformity $d$, if 
	\begin{equation*}
		d=\max_{a \in \F_{2^n}^*,b\in \F_{2^n}} |\{x \colon F(x)+F(x+a)=b\}|.
	\end{equation*}
	A function with differential uniformity $2$ is called almost perfect nonlinear (APN) on $\F_{2^n}$.
\end{definition}
To resist differential attacks, a vectorial Boolean function should have low differential uniformity. As the differential uniformity is always even, the APN functions yield the best resistance against 
differential attacks.

\begin{definition}
	The Walsh transform $W_F:\F_{2^n} \times \F_{2^n} \to \mathbb{Z}$ of a function $F \colon \F_{2^n} \rightarrow \F_{2^n}$ is defined as follows:
	\begin{equation*}
		W_F(a,b)=\sum_{x\in \F_{2^n}} (-1)^{\Tr(aF(x)+bx)}.
	\end{equation*}
	The nonlinearity of $F$ is defined as 
	\begin{equation}
		nl(F) = 2^{n-1}-\frac{1}{2} \max_{a \in \F_{2^n}^*,b\in \F_{2^n}} |W_F(a,b)|.
	\end{equation}
\end{definition}
The higher the nonlinearity of a vectorial Boolean function, the better is its resistance to linear attacks.

There are several operations on the set of Boolean functions under which linear and differential
properties are invariant. They lead to several equivalence concepts for vectorial Boolean functions.
We denote by
\begin{equation*}
G_F=\{(x,F(x)) \colon x \in \F_{2^n}\} \subset \F_{2^n} \times \F_{2^n}
\end{equation*}
the graph of the function $F \colon \F_{2^n}\rightarrow \F_{2^n}$.
 In the next definition and
in the remainder of the paper we use 
the term linear function to refer to an $\F_2$-linear one. {Similarly, we will call a function affine if it is sum of a linear function and a constant.}

\begin{definition} \label{def:equiv}
{
	Two functions $F_1,F_2 \colon \F_{2^n}\rightarrow \F_{2^n}$ are called \emph{extended affine equivalent (EA-equivalent)} if there are affine permutations $A_1,A_2$ and an affine mapping $A_3$ mapping from $\F_{2^n}$ to itself such that 
\begin{equation}
A_1(F_1(A_2(x)))+A_3(x)=F_2(x).
\label{eq:ea}
\end{equation}
$F_1$ and $F_2$ are called \emph{affine equivalent} if they are EA-equivalent and it is possible to choose $A_3=0$ in Eq.~\eqref{eq:ea}. }

	Moreover, $F_1$ and $F_2$ are called \emph{CCZ-equivalent} if there are linear  functions $\alpha,\beta,\gamma,\delta \colon \F_{2^n} \rightarrow \F_{2^n}$ and $a,b \in \F_{2^n}$ such that $\mathcal L \colon \F_{2^n}^2 \rightarrow \F_{2^n}^2$
	defined by 
	\begin{equation*}
		\mathcal L (x,y)=(\alpha(x)+\beta(y),\gamma(x)+\delta(y))
	\end{equation*}
	is bijective and 
	$$
	\mathcal{L}(G_{F_1})+(a,b) = G_{F_2}.
	$$
	
$F_1$ and $F_2$ are EA-equivalent if and only if a mapping $\mathcal{L}$ defined as above can be found with $\beta = 0$, and affine equivalent if and only if a mapping $\mathcal{L}$ can be found with $\beta = \gamma = 0$.

\end{definition}

The concept of CCZ-equivalence was introduced in \cite{ccz} in 1998.
It has been extensively studied, since it is a powerful tool for constructing and 
studying cryptological functions \cite{eaccz,cczeagraph,twisting,twisting2}.
Clearly, affine equivalence implies EA-equivalence, which in turn implies CCZ-equivalence. 
 The size of the image set is invariant under affine equivalence
 but in general it is changed under EA-equivalence. Nonlinearity and differential uniformity are invariant under CCZ-equivalence.  \\

\textbf{Outline.} In this paper, we consider  EA- and CCZ-equivalence to the inverse function. This is a particularly interesting case because of the good cryptographic properties of the inverse function. In the second section, we show that some questions about CCZ- and EA-equivalence to a function $F$ are related to the existence of permutations  of the form $L_1(F(x))+L_2(x)$. Accordingly, we investigate the existence of permutations of the form $L_1(x^{-1})+L_2(x)$. This problem is related to Kloosterman zeros, i.e. elements whose Kloosterman sum is zero. In Section~\ref{sec:kloos} we give an upper bound on the maximal size of a subspace of $\F_{2^n}$ that contains only Kloosterman zeros.  Using this result, we show in Section~\ref{sec:last} that there are no permutations of the form $L_1(x^{-1})+L_2(x)$ if $\ker(L_1)$ or $\ker(L_2)$ is large.

\section{EA- and CCZ-equivalence and specific permutations }

The only known examples of APN permutations on $\F_{2^n}$ with $n$ even are
constructed for $n=6$ via study of the set of CCZ-equivalent functions to a known non-bijective function
in \cite{dillon-perm}. The question about existence of APN permutations 
for an even $n \geq 8$ is considered as the biggest challenge in the research on APN functions. As the examples in \cite{dillon-perm} suggest, a better understanding of CCZ-equivalence for permutations
could be essential for progressing on this topic. Proposition \ref{prop:start} shows that this is
closely related to study of permutations of form $L_1(F(x)) + L_2(x)$ with linear $L_1, L_2$.
We would like to note that similar results  are mentioned in various papers,
for instance in \cite{twisting,twisting2,eaccz}.\\


\begin{proposition} \label{prop:start} 
\begin{description}
	\item[(a)] Let $F \colon \F_{2^n} \rightarrow \F_{2^n}$ and no permutation of the form $F(x)+L(x)$ exist with non-zero linear $L(x)$. Then every permutation that is EA-equivalent to $F$ is already affine equivalent to it. In particular, if such an $F$ is not bijective, then there are no  EA-equivalent 
	permutations to $F$.
		\item[(b)] Let $F \colon \F_{2^n} \rightarrow \F_{2^n}$ and no permutation of 
		the form $L_1(F(x))+L_2(x)$ exist with non-zero linear $L_1,L_2$. 
		Then every function that is CCZ-equivalent to $F$ is EA-equivalent to 
		$F$ or $F^{-1}$ (if it exists). Moreover, all permutations that are CCZ-equivalent to $F$ are affine equivalent to $F$ or $F^{-1}$.
	\end{description}
	
\end{proposition}
\begin{proof}
	(a) Let $F_2$ be a permutation EA-equivalent to $F$. By the definition of EA-equivalence, there exist $(a,b) \in \F_{2^n}^2$ and  a bijective mapping $\mathcal{L} \colon \F_{2^n}^2 \rightarrow \F_{2^n}^2$ defined by 
	$\mathcal{L}(x,y) =(\alpha(x),\gamma(x)+\delta(y))$ with linear functions $\alpha,\gamma,\delta \colon \F_{2^n} \rightarrow \F_{2^n}$ such that 
	\begin{equation*}
		\mathcal{L}(x,F(x))+(a,b) = (\alpha(x)+a,\gamma(x)+\delta(F(x))+b)=(\pi(x),F_2(\pi(x)))
	\end{equation*}
	where $\pi \colon \F_{2^n} \rightarrow \F_{2^n}$ is the permutation given by $\pi(x) = \alpha(x)+a.$
	Note that the function $\delta$ is bijective on $\F_{2^n}$, since 
	 $\mathcal{L}$ is bijective on $\F_{2^n}^2$. Also the composition  
	 $F_2(\pi(x))$ is bijective on $\F_{2^n}$, implying that $\gamma(x)+\delta(F(x))$ is bijective, and hence also  $\delta^{-1}(\gamma(x))+F(x)$ is a permutation. Since $\delta^{-1}(\gamma(x))$ is linear, our  assumption on $F$ yields that $\gamma=0$, completing the proof. \\
	
(b) Let now $F_2$ be a function CCZ-equivalent to $F$. By the definition of CCZ-equivalence, there 
exist $(a,b) \in \F_{2^n}^2$ and  a bijective mapping $\mathcal{L} \colon \F_{2^n}^2 \rightarrow \F_{2^n}^2$
given by $\mathcal{L}(x,y)=(\alpha(x)+\beta(y),\gamma(x)+\delta(y))$ with linear $\alpha,\beta,\gamma,\delta \colon \F_{2^n} \rightarrow \F_{2^n}$ such that
		\begin{align*}
		\mathcal{L}(x,F(x))+(a,b) &= (\alpha(x)+\beta(F(x))+a,\gamma(x)+\delta(F(x))+b)\\
		&=(\pi(x),F_2(\pi(x)))
	\end{align*}
		where $\pi \colon \F_{2^n} \rightarrow \F_{2^n}$ is the permutation on $\F_{2^n}$ given
		by $\pi(x) = \alpha(x)+\beta(F(x))+a$.  By our assumption on $F$, either $\alpha =0$ or $\beta =0$. 
		 Assume first that $\alpha = 0$. Then $\pi(x)= \beta(F(x))+a$ and in particular both $F$ and $\beta$ are bijective. Further, $\gamma$ is bijective since $\mathcal{L}$ is bijective. We then have 
		$$\gamma(x)+\delta(F(x))+b = F_2(\pi(x)) = F_2(\beta(F(x))+a).$$
		The composition with the inverse  $F^{-1}(x)$ yields
		$$\gamma(F^{-1}(x))+\delta(x)+b = F_2(\beta(x)+a),$$ 
		and hence $F_2$ is EA-equivalent to $F^{-1}$. 
		In the case $\beta = 0$ we get similarly $\pi(x) = \alpha(x)+a$ and 
		$$\gamma(x)+\delta(F(x))+b = F_2(\pi(x)) = F_2(\alpha(x)+a),$$ where the mappings $\alpha$ and $\delta$ are bijective. Hence $F_2$ is EA-equivalent to $F$. 
		
		Now assume that $F_2$ is additionally a permutation. If $F_2$ is EA-equivalent to $F$ then $F_2$ is affine equivalent to $F$ using the statement in (a). 
		Let us now consider the case that $F_2$ is EA-equivalent to $F^{-1}$. Observe that $F^{-1}(x)+L(x)$ is a permutation if and only if $L(F(x))+x$ is a permutation, so there are no permutations of the form $F^{-1}(x)+L(x)$ by the assumption stated in the proposition. Again using (a), we conclude that $F_2$ is affine equivalent to $F^{-1}$.

	\qed
\end{proof}

 The following proposition gives a criterion when a function $L_1(F(x))+L_2(x)$ is bijective. For a linear mapping $L$, we denote by $L^*$ its adjoint mapping with respect to the bilinear form 
\begin{equation*}
\langle x,y \rangle =\Tr(xy)
\end{equation*}

where $\Tr$ is the absolute trace mapping, i.e. we have 
\begin{equation*}
\Tr(L(x)y)=\Tr(xL^*(y))
\end{equation*}
for all $x,y \in \F_{2^n}$. Further, for a subset $A \subseteq \F_{2^n}$ we denote by $A^\perp$
its orthogonal complement, that is
\begin{equation*}
A^\perp=\{x \in \F_{2^n} \colon \Tr(ax)=0 \text{ for all } a \in A\}.
\end{equation*}

\begin{proposition} \label{prop:walshzeroes}
	Let $F \colon \F_{2^n} \rightarrow \F_{2^n}$ and $L_1,L_2$ be linear mappings. The function $L_1(F(x))+L_2(x)$ is a permutation if and only if
	\begin{equation*}
		W_F(L_1^*(b),L_2^*(b)) =0
	\end{equation*}
	for all $b \in \F_{2^n}^*$.
\end{proposition}
\begin{proof}
	It is well-known that a function is a permutation if and only if all of its component functions are balanced (for a proof, see \cite[Theorem 7.7]{LN}). Consequently, $L_1(F(x))+L_2(x)$ is a permutation if and only if
	\begin{align*}
		0&=\sum_{x \in \F_{2^n}} (-1)^{\Tr(b(L_1(F(x))+L_2(x)))} \\
		&=\sum_{x \in \F_{2^n}} (-1)^{\Tr(L_1^*(b)F(x)+L_2^*(b)x)} = W_F(L_1^*(b),L_2^*(b))
	\end{align*}
	for all $b \in \F_{2^n}^*$.
		\qed
\end{proof}

Permutations of form $L_1(F(x))+L_2(x)$ are characterized for some special choices of $F$ and $L_1,L_2$.
It was shown in \cite{charpinpasalic} that no permutation of the form $x^d+L(x)$ exists when there is an $a \in \F_{2^n}$  such that $\Tr(ax^d)$ is bent. 
Corollary 2.3 from \cite{gerike-kyureg}  implies that 
 $x^d + L(K(x))$ is not bijective on $\F_{q}$ for an arbitrary function $K$ whenever $\gcd(d, q-1) \ne 1$
and $L$ is a non-bijective linear funtion.
In  \cite{eapoly} a characterization of all permutations of the form $x^{2^i+1}+L(x)$ over $\F_{2^n}$ with $\gcd(i,n)=1$ was given, as well as some results for the more general case $x^d+L(x)$. Permutations of the form $x^{2^i+1}+L(x)$ over $\F_{2^n}$ with $\gcd(i,n)>1$ were recently considered in \cite{leoboomerang}. 
A particularly interesting  case are the functions of shape $L_1(x^{-1})+L_2(x)$ because of their good cryptographic properties. 
(Here we use as usual the convention $0^{-1}=0$.) 
 It was shown in  \cite{oddchar} that such functions are never permutations in characteristic $\geq 5$ (except for the trivial cases $L_1=0$ or $L_2=0$). In characteristic $3$, no permutations of the type $x^{-1}+L(x)$ with $L\neq 0$ exist, except for sporadic cases in the small fields $\F_3$ and $\F_9$. In this paper we are interested in the case of characteristic 2. If $L_1$ or $L_2$ is bijective, then 
 $L_1(x^{-1})+L_2(x)$ cannot be bijective  on $\F_{2^n}$ for $n\geq 5$ as shown in \cite{eainverse}.

\begin{theorem}[\cite{eainverse}] \label{thm:chin}
	Let $F\colon \F_{2^n} \rightarrow \F_{2^n}$ be defined by  $F(x)= x^{-1}+L(x)$ with some linear mapping $L(x) \neq 0$. If $n\geq 5$ then $F$ is not a permutation.
\end{theorem} 

The following result is an immediate consequence of Theorem \ref{thm:chin}.
\begin{corollary}\label{thm:chin-gen}
	Let $n\geq 5$ and $F\colon \F_{2^n} \rightarrow \F_{2^n}$ be defined by  $F(x) = L_1(x^{-1})+L_2(x)$,
	where $L_1, L_2$ are non-zero linear functions of $\F_{2^n}$. If $L_1$ or $L_2$ is bijective, then $F$ is not a permutation on $\F_{2^n}$.
\end{corollary} 
\begin{proof} Note that $F(x)$ is bijective if and only if $F(x^{-1}) = L_1(x)+L_2(x^{-1})$ is so.
Hence without loss of generality suppose $L_1$ is bijective. Then the composition
$L_1^{-1}(F(x)) = x^{-1} + L_1^{-1}(L_2(x))$ is bijective if and only if $F$ is so, and
 Theorem \ref{thm:chin} completes the proof. 
\qed
\end{proof}

In this paper, we 
continue the study of functions  $L_1(x^{-1})+L_2(x)$ where $L_1,L_2$ are linear polynomials over $\F_{2^n}$. In the case of the inverse function $x\mapsto x^{-1}$, the Walsh transform is closely connected to Kloosterman sums. 
\begin{definition}
	For $a \in \F_{2^n}$, the Kloosterman sum of $a$ over $\F_{2^n}$ is defined as
	\begin{equation*}
		K_n(a)=\sum_{x \in \F_{2^n}}(-1)^{\Tr(x^{-1}+ax)}.
	\end{equation*}
	An element $a \in \F_{2^n}$ with $K_n(a)=0$	is called a Kloosterman zero.
\end{definition}

Note $K_n(a)=W_F(1,a)$ for $F(x)=x^{-1}$. More precisely, for $a\ne 0$ we have 
\begin{equation*}
W_F(a,b)=\sum_{x \in \F_{2^n}} (-1)^{\Tr(ax^{-1}+bx)}=\sum_{x \in \F_{2^n}} (-1)^{\Tr(x^{-1}+abx)}=K_n(ab)
\end{equation*}
using the substitution $x \mapsto ax$. For $a=0$ and $b \neq 0$, we have $K_n(ab)=W_F(a,b)=0$. \\

Proposition~\ref{prop:walshzeroes} can thus be stated using Kloosterman sums:
\begin{corollary} \label{cor:inverse}
	Let  $L_1,L_2$ be linear functions of $\F_{2^n}$. Then $L_1(x^{-1})+L_2(x)$ is a permutation on $\F_{2^n}$ if and only if
	$\ker(L_1^*)\cap \ker(L_2^*)=\{0\}$ and
		\begin{equation*}
		K_n(L_1^*(b)L_2^*(b)) =0
	\end{equation*}
	for all $b \in \F_{2^n}$.
\end{corollary}
\begin{proof}
	By Proposition~\ref{prop:walshzeroes}, $L_1(x^{-1})+L_2(x)$ is a permutation if and only if $W_{F}(L_1^*(b),L_2^*(b)) =0$ for all $b \neq 0$. If $b \in \ker(L_1^*)\cap \ker(L_2^*)$, then $W_{F}(L_1^*(b),L_2^*(b))=2^n\neq 0$. In the other cases $W_{F}(L_1^*(b),L_2^*(b)) = K(L_1^*(b)L_2^*(b))$ by the considerations above.
	\qed
\end{proof}

Corollary~\ref{cor:inverse} shows that a function $L_1(x^{-1}) + L_2(x)$ is bijective on $\F_{2^n}$
only if 
the set $\{L^*_1(x)L^*_2(x) | x \in \F_{2^n} \}$ is a subset of the set of Kloosterman zeroes. Conversely,
in \cite{hx-ks} specific functions of shape $L_1(x^{-1}) + L_2(x)$ are used
to obtain  identities for Kloosterman sums.

\section{Vector spaces of Kloosterman zeros} \label{sec:kloos}

The Kloosterman sums provide a powerful tool for studying {additive} properties of the
inversion on finite fields. Kloosterman zeros are used  for the construction of bent and hyperbent functions (see for example  \cite{dillonthesis,hyperbent,vecbent}). Vector spaces of Kloosterman zeros of dimension $d$ in $\F_{2^n}$ can be used to construct vectorial bent functions from $\F_{2^{2n}}$ to $\F_{2^d}$ by modifying Dillon's construction, as shown in \cite[Proposition 5]{vecbent}.\\

 Few results about the distribution of Kloosterman zeros are known. There is a way to compute the number of Kloosterman zeros \cite{goppa}, which relies on determining the class number of binary quadratic forms. However, it is difficult to use this method to derive a theoretical result on the number and distribution of Kloosterman sums. It was shown that for all $n$, Kloosterman zeros exist \cite{lachaud} (note that this is not true in characteristic $ \geq 5$  \cite{finns}). Moreover, it is known that for $n>4$, Kloosterman zeros are never contained in proper subfields of $\F_{2^n}$ \cite{zeros}. In \cite{shpar}, it is noted that $|\{a\in \F_{2^n} \colon K_n(a)=0\}|=O(2^{3n/4})$. In this section, we give an upper bound for the size of vector spaces that contain exclusively Kloosterman zeros. \\

Let $B$ be a bilinear form from $\F_{2^n}$ to $\F_2$. We denote by $\rad(B)=\{y \in \F_{2^n} \colon B(x,y)=0 \text{ for all } x \in \F_{2^n}\}$ the radical of $B$. Given a quadratic form $f:\F_{2^n} \to \F_2$, let
 $B_f(x,y)=f(x)+f(y)+f(x+y)$ be the bilinear form associated to it. 
The radical of the quadratic form $f$ is defined as $\rad(B_f) \cap f^{-1}(\{0\}).$
A quadratic form is called non-degenerate if $\rad(f)=\{0\}$. \\

Let $Q \colon \F_{2^n} \rightarrow \F_2$ be the quadratic form defined by 
\begin{equation*}
Q(x)=\sum_{0 \leq i <j <n}x^{2^i+2^j}
\end{equation*}
for all $x \in \F_{2^n}$. Note that if $m_a$ is the minimal polynomial of $a \in \F_{2^n}$ over $\F_2$ of degree $d$, then $Q(a)$ is the third coefficient of $m_a^{n/d}$. Indeed, recall that
 $m_a^{n/d} = \chi_a$ is the characteristic polynomial of $a$ over $\F_{2}$ and $\chi_a(x)=\sum_{i=0}^{n-1}(x+a^{2^i})$. By expanding the product, we see that $Q(a)$ is the coefficient of $x^{n-2}$ as claimed. This in particular shows that $Q(a) \in \F_2$. \\

The dyadic approximation of Kloosterman sums are often used to study Kloosterman zeroes.
A nice survey on this topic is given in \cite{zinoviev-survey}.
The main tool for results in this section is the following characterization 
of Kloosterman sums divisible by  $2^4$.

\begin{theorem}[\cite{faruk}] \label{thm:faruk}
	Let $n\geq 4$ and $a \in \F_{2^n}$. Then $K_{n}(a) \equiv 0 \pmod {16}$ if and only if $\Tr(a)=0$ and $Q(a)=0$. 
\end{theorem}

{Theorem~\ref{thm:faruk} implies that the Kloosterman zeroes are contained
in the intersection of the quadric $\{x \in \F_{2^n} \colon Q(x) = 0\}$ and the hyperplane 
 $$H=\{x\in\F_{2^n} \colon \Tr(x)=0\}.$$  
Therefore we consider the quadratic form $Q|_H$ which is induced by $Q$ on $H$. 
We first determine its radical.}

\begin{lemma} \label{lem:radical}
We have 
	\begin{equation*}
		\rad(Q|_H)=\begin{cases}
							\{0,1\}, & n \equiv 0 \pmod 4 \\
							\{0\}, &\text{ else.} 
							\end{cases}
	\end{equation*}
\end{lemma}
\begin{proof}
	First we compute the bilinear form associated to $Q$:
	\begin{align*}
			B_Q(x,y)&=\sum_{0 \leq i <j <n}x^{2^i+2^j}+\sum_{0 \leq i <j <n}y^{2^i+2^j}+\sum_{0 \leq i <j <n}(x+y)^{2^i+2^j} \\
			&=\sum_{i\neq j}x^{2^i}y^{2^j}=\sum_{i=0}^{n-1}x^{2^i}\sum_{j \neq i }y^{2^j}\\
			&=\sum_{i=0}^{n-1}x^{2^i}(\Tr(y)+y^{2^i})=\sum_{i=0}^{n-1}(xy)^{2^i}+\Tr(y)\sum_{i=0}^{n-1}x^{2^i}\\
			&=\Tr(xy)+\Tr(x)\Tr(y)=\Tr((y+\Tr(y))x).
	\end{align*}
	{Since $\Tr(y)=0$ for all $y \in H$, we have
	$$
	B_{Q|_H}(x,y) = \Tr(xy).
		$$
	Then $y \in \rad(B_{Q|_H})$, if $B_{Q|_H}(x,y)= \Tr(xy) =0$ for all $x \in H$. 
	Hence $\rad(B_{Q|_H}) = \F_2 \cap H$.
	 Observe that $1 \in H$ if and only if $n$ is even, so $\rad(B_{Q|_H})=\{0\}$ if $n$ is odd and $\rad(B_{Q|_H})= \F_2$ if $n$ is even.} One can easily verify that 
	\begin{equation*}
		Q(1)=\frac{n(n-1)}{2}=\begin{cases}
			0 & n \equiv 0,1 \pmod 4 \\
			1 & n \equiv 2,3 \pmod 4
		\end{cases}
	\end{equation*}
	 and the result follows.
		\qed
\end{proof}

Let $N(Q|_H(x)=u)$ denote the number of solutions of $Q|_H(x)=u$ for $u\in \F_2$. Observe that $N(Q|_H(x)=0)$ is precisely the number of elements $x \in \F_{2^n}$ whose second and third coefficients of the characteristic polynomial $\chi_x$ are zero. The value $N(Q|_H(x)=a)$ was investigated in \cite{prescr1,prescr2,prescr3}, where  irreducible polynomials with prescribed coefficients were studied. In particular, the value $N(Q|_H(x)=0)$ was determined. We summarize some of their results in the following theorem.
\begin{theorem}\label{thm:solutions}
	Let $N(Q|_H(x)=0)$ be the number of $x \in H$ with $Q|_H(x)=0$. Then $N(Q|_H(x)=0)=2^{n-2}+e$ where
	\begin{equation*}
		e=\begin{cases}
			-2^{\frac{n-2}{2}}, &n \equiv 0\pmod 8 \\
			2^{\frac{n-3}{2}}, &n \equiv 1,7\pmod 8 \\
			0, &n \equiv 2,6\pmod 8 \\
			-2^{\frac{n-3}{2}}, &n \equiv 3,5\pmod 8 \\
			2^{\frac{n-2}{2}}, &n \equiv 4\pmod 8.
		\end{cases}
	\end{equation*}
\end{theorem} 

Two quadratic forms $f$ and $g$ on a vector space $V$ are called equivalent if $f$ can be transformed into $g$ with a non-singular linear transformation of $V$.
	The following result is well known (see e.g. \cite{LN,Hou}).
	
\begin{theorem}[Classification of quadratic forms]\label{thm:class}
	Let $f \colon V \rightarrow \F_2$ with $\dim(V)=n$ be a quadratic form with $\dim(\rad(f))=w$.
	Then $f$ is equivalent to one of three forms:
	\begin{align*}
		f &\simeq \sum_{i=1}^v x_iy_i & \text{(hyperbolic case)}\\
		f &\simeq z+\sum_{i=1}^v x_iy_i & \text{(parabolic case)}\\
		f &\simeq x_1^2+x_1y_1+y_1^2+\sum_{i=2}^v x_iy_i & \text{(elliptic case)},
	\end{align*}
	where $v=\lfloor (n-w)/2 \rfloor$. \\

	The value of $N(f(x)=0)$ depends only on $n$, $w$ and the type of the quadratic form. More precisely,
	\begin{equation*}
		N(f(x)=0) = 2^{n-1}+\Lambda(f)2^{\frac{n+w-2}{2}},
	\end{equation*}
	with 
	\begin{equation*}
		\Lambda(f)=\begin{cases}
			1, &\text{ if }f \text{ is hyperbolic}\\
			0,  &\text{ if }f \text{ is parabolic}\\
			-1, & \text{ if }f \text{ is elliptic}.
		\end{cases}
	\end{equation*}

\end{theorem}

The \emph{Witt index} of a quadratic form is the number of pairs $x_iy_i$ that appear in the decomposition described above. In particular, the Witt index of $f$ is $v$ in the hyperbolic and parabolic case, and $v-1$ in the elliptic case.

\begin{remark}
Just using the classification of quadratic forms in Theorem~\ref{thm:class} and the determination of the radical in Lemma~\ref{lem:radical} we can give a simple alternative proof of the cases $n \equiv 2,6 \pmod 8$ in Theorem~\ref{thm:solutions}. Indeed, in these cases $Q|_H$ is necessarily parabolic which immediately gives the value for $N(Q|_H(x)=0)$.
\end{remark}

{ We are now interested in the maximal dimension of a subspace contained in a quadric. Let $f$ be
a quadratic form on $V$.
A  subspace $W$ of $V$ is called totally isotropic if  $f(w)=0$ for all $w \in W$. And a subspace 
$W$ is called maximal totally isotropic
 if there is no subspace $W_2$ with $f(w)=0$ for all $w \in W_2$ and $W \subsetneq W_2 \subseteq V$.
Any two maximal 
 totally isotropic subspaces have the same dimension, which  {is the sum of  the Witt index
and the dimension of the radical of
the quadratic form,} as the following result implies.}

\begin{proposition}[{\cite[Corollary 4.4.]{quadforms}}] \label{prop:singular}
	Let $f \colon V \rightarrow \F_2$ be a non-degenerate quadratic form on a vector space $V$ over $\F_2$ with $\dim(V)=n$. Let $W$ be a maximal totally isotropic subspace of $V$. Then, the dimension of $W$ is equal to the Witt index of $f$. In particular, we have
	\begin{equation*}
		\dim(W) = \begin{cases}
				\frac{n}{2}, &\text{ if }f \text{ is hyperbolic} \\
				\frac{n-1}{2}, &\text{ if }f \text{ is parabolic} \\
				\frac{n-2}{2}, &\text{ if }f \text{ is elliptic}.
		\end{cases}
	\end{equation*}
\end{proposition}

{We collect the above observations to give a {sharp} upper bound on the size of vector spaces that consist of elements with Kloosterman sum divisible by $16$.}

\begin{proposition}\label{prop:mod16}
	Let $W$ be a subspace of $\F_{2^n}$ with $K_{n}(w)\equiv 0 \pmod{16}$ for all $w \in W$ and $n\geq 5$. Then $\dim W \leq d$ where
	\begin{equation*}
		d = \begin{cases}
			\frac{n-2}{2}, &n \equiv 0,2,6\pmod 8 \\
			\frac{n-1}{2}, &n \equiv 1,7\pmod 8 \\
			\frac{n-3}{2}, &n \equiv 3,5\pmod 8\\
			\frac{n}{2}, &n \equiv 4\pmod 8.
		\end{cases}
	\end{equation*}
	The bounds are sharp.
\end{proposition}
\begin{proof}
From the Theorems~\ref{thm:solutions} and \ref{thm:class} we deduce that $Q|_H$ is elliptic if $n \equiv 0,3,5 \pmod 8$, hyperbolic if $n \equiv 1,4,7 \pmod 8$ and parabolic if $n \equiv 2,6 \pmod 8$. 
	In the cases $n \not\equiv 0,4 \pmod 8$ the quadratic form $Q|_H$ is non-degenerate by Lemma~\ref{lem:radical} and we immediately get bounds on $\dim(W)$ from Proposition~\ref{prop:singular} (recall that $Q|_H$ is a quadratic form on an $(n-1)$ dimensional space). If $n \equiv 0,4 \pmod 8$ then $\dim(\rad(Q|_H))=1$, so $\dim V \leq 1+\frac{n-4}{2}=\frac{n-2}{2}$ if $n \equiv 0 \pmod 8$ and $\dim V \leq 1+\frac{n-2}{2}=\frac{n}{2}$ if $n \equiv 4 \pmod 8$.
	\qed
\end{proof}

\begin{remark} {Every vector space $W$ that contains exclusively Kloosterman zeros is of course also a vector space that contains only Kloosterman sums divisible by $16$. In particular, by Propositions~\ref{prop:singular} and \ref{prop:mod16}, all vector spaces of Kloosterman zeros are necessarily contained in a {maximal
totally isotropic} vector space of $Q|_H$. However, these vector spaces are generally not unique.}
\end{remark}

{ Using Proposition~\ref{prop:mod16}, we get the following result.}

\begin{theorem} \label{thm:kloozeros}
	Let $W$ be a subspace of $\F_{2^n}$ such that $K_{n}(v)=0$ for all $v \in W$ and $n\geq 5$. Then $\dim W \leq d$ where
	\begin{equation*}
		d = \begin{cases}
			\frac{n-2}{2}, &n \equiv 0,2,4,6\pmod 8 \\
			\frac{n-1}{2}, &n \equiv 1,7\pmod 8 \\
			\frac{n-3}{2}, &n \equiv 3,5\pmod 8.
		\end{cases}
	\end{equation*}

\end{theorem}
\begin{proof} {
	The bound follows from Proposition~\ref{prop:mod16} for all cases except $n \equiv 4 \pmod 8$.
	In the latter case the bound of Proposition ~\ref{prop:mod16} can be improved by one using the following
	observation for even $n$.
	\footnote[1]{This is due to an anonymous referee.}
	Let $n=2k$ be even. As noted in \cite{zeros}, there are no non-zero Kloosterman zeros in the  subfield $\F_{2^{k}}$. We have $\F_{2^k}\subset H$,  $W\subset H$ and $W \cap \F_{2^k}=\{0\}$, implying $\dim(V)\leq \frac{n-2}{2}$.   }
		\qed
\end{proof}

{We would like to mention that the following approach yields a slightly weaker bound than the one given in Theorem~\ref{thm:kloozeros}. } The following identity for sums of Kloosterman sums over a vector space was given in \cite[Proposition 3]{charp_kloos}: For any subspace $V$ of $\F_{2^n}$ with $\dim(V)=k$ we have
\begin{equation*}
	\sum_{a \in V}(K_n^2(a)-K_n(a))=2^{n+k}-2^{n+1}+2^k\sum_{u \in V^{\perp}}K_n({u^{-1}}).
\end{equation*}
If $V$ contains exclusively Kloosterman zeros, we get 
\begin{equation*}
	0=2^{n+k}-2^{n+1}+2^k\sum_{u \in V^{\perp}}K_n(u^{-1}),
\end{equation*}
recall we set $0^{-1} =0$. Bounding the Kloosterman sum in the right hand side of the equation using the Weil bound $|K_n(a)|\leq 2^{\frac{n}{2}+1}$, we get
\begin{equation*}
	0\geq 2^{n+k}-2^{n+1}-2^k2^{n-k}2^{\frac{n}{2}+1}=2^{n+k}-2^{n+1}-2^{\frac{3n}{2}+1}.
\end{equation*}
This shows that $k=\dim(V)\leq \frac{n}{2}+1$ for $n\geq 3$. \\

{
\begin{remark}\label{rem:bound}
Theorem~\ref{thm:kloozeros} provides to our knowledge the first general upper bound on the maximal size of subspaces of Kloosterman zeros. However, experimental results indicate that our bound is weak, see 
Table~\ref{t:zeros}. Our bound is sharp for very small $n$ (see right table in Table~\ref{t:zeros}), which is not surprising since the approximation modulo $16$ is strong for small $n$. Numerics in Table~\ref{t:zeros} were computed using \cite{goppa,mth-pl} for the  left table and \cite{perrin} for the right
table. The left table shows that the total number of Kloosterman zeros in the field $\F_{2^n}$ is close to $2^{n/2}$ for $n\leq 60$. It is of course not to expect that the set of Kloosterman zeros has a strong additive structure, so we believe that the bound of Theorem~\ref{thm:kloozeros} can be significantly improved. 
\end{remark} }
\begin{table}[ht]
	\centering
	\begin{tabular}{ ||c|c|| } 
	 \hline
		$n$ & $2^{\frac{-n}{2}} \mathcal{Z}(n)$\\
	 \hline  \hline
	5 & 0.88 \\
	 10 & 1.87 \\
	 15 & 1.57 \\
	20 & 0.86\\
	25 & 0.67\\
	30 & 1.29 \\
	35 & 1.15 \\
	40 & 1.15 \\
	45 & 1.14 \\
	50 & 0.91 \\
	55 & 1.32 \\
	60 & 1.25 \\
		\hline
	\end{tabular}
	\quad\quad\quad\quad
\begin{tabular}{||c|c||}
	\hline
		$n$ & $\dim(V)$\\
 \hline  \hline
	 5 & 1 \\
	 6 & 2 \\
	7 & 3\\
	8 & 1\\
	9 & 1 \\
	10 & 2 \\
	11 & 2 \\
	12 & 2 \\
	13 & 1 \\
	14 & 3 \\
	15 & 4 \\
	16 & 2 \\
		\hline 
\end{tabular}
\vspace*{0.3cm}
	\caption{Left Table: Comparison of the number of Kloosterman zeros over $\F_{2^n}$ to the value $2^{n/2}$. Here, $\mathcal{Z}(n)$ denotes the number of Kloosterman zeros over $\F_{2^n}$.\protect\\ 
	Right table: the maximal dimension of a subspace $W$ of $\F_{2^n}$ that contains exclusively Kloosterman zeros.}
	\label{t:zeros}
\end{table}

\begin{problem}
	Find a better bound on the maximal size of a subspace containing exclusively Kloosterman zeros.
\end{problem}

\section{Permutations of the form $L_1(x^{-1})+L_2(x)$} \label{sec:last}

We now apply the results from the previous section. 
The following lemma is well-known. We include a simple proof of it for the convenience of the reader.

\begin{lemma}
	Let $L \colon \F_{2^n} \rightarrow \F_{2^n}$ be  linear and $L^*$ be its adjoint mapping. 
	Then $\dim(\im(L^*))=\dim(\im(L))$ and $\dim(\ker(L^*))=\dim(\ker(L))$.
\end{lemma}
\begin{proof}
	Let $v \in \im(L^*)$ and $w \in \ker(L)$. We can write $v=L^*(x)$ for some $x \in \F_{2^n}$. Then 
	$\langle v,w \rangle = \langle L^*(x),w \rangle = \langle x,L(w)\rangle = \langle x,0 \rangle =0$, so $\im(L^*)\subseteq\ker(L)^\perp$, in particular $\dim(\im(L^*)) \leq \dim(\im(L))$. The other inequality holds with $L^{**}=L$.
	
	The statement on the kernel follows from $\dim(\im(L))+\dim(\ker(L))=n$.
		\qed
\end{proof}

Corollary \ref{thm:chin-gen} shows that a function $L_1(x^{-1})+L_2(x)$ cannot be
bijective on $\F_{2^n}$ if at least one of $L_1$ or $L_2$ is bijective, equivalently has
a trivial kernel. The next result shows that such a function is not bijective also in the case
when the kernel of $L_1$ or $L_2$ is large.

\begin{theorem}
	Let $n\geq 5$ and $F(x)=L_1(x^{-1})+L_2(x)$ where $L_1$ and $L_2$ are 
	non-bijective non-zero linear functions of  $\F_{2^n}$. Further, let  $d$ be defined as in Theorem~\ref{thm:kloozeros}.
	If $\max(\dim(\ker(L_1)),\dim (\ker(L_2)))> d$, then $F$ does not permute $\F_{2^n}$.
	\end{theorem}
\begin{proof}
	Observe that $F(x)$ is a permutation if and only if $F(x^{-1})=L_1(x)+L_2(x^{-1})$ is so. 
	Hence we may assume without loss of generality that $\dim (\ker(L_1))\geq \dim (\ker(L_2)) \geq 1$. 
	Suppose $F$ is a permutation. Then
by Corollary~\ref{cor:inverse} we have $\Ker(L_1^*)\cap\Ker(L_2^*)=\{0\}$ and \\$K_{n}(L_1^*(b)L_2^*(b))=0$ for all $b \in \F_{2^n}$.
	Set $e=\dim \ker L_1=\dim \ker L_1^*$. Choose $0\neq c \in \ker(L_2^*)$. The set 
	\begin{equation*}
		V=L_1^*(c+\ker(L_1^*))\cdot L_2^*(c+\ker(L_1^*))=L_1^*(c)\cdot L_2^*(\ker(L_1^*))
	\end{equation*}
	is a vector space that is contained in the image set of $L_1^*(b)L_2^*(b)$. In particular $K_n(v)=0$ for all $v \in V$. Since  $\Ker(L_1^*)\cap\Ker(L_2^*)=\{0\}$ we have $\dim(V)=e$.  Theorem~\ref{thm:kloozeros} then implies that $e \leq d$.
		\qed
\end{proof}

We conjecture that the following statements hold: \footnote{After the acceptance of this submission,
Lukas K\"olsch found a proof for Conjecture\,\ref{conj-open}.}
\begin{conjecture} \label{conj-open}
	Let $F=L_1(x^{-1})+L_2(x)$ where $L_1 \neq 0 $ and $L_2 \neq 0$ are linearized polynomials over $\F_{2^n}$ with $n\geq 5$. Then $F$ does not permute $\F_{2^n}$.
\end{conjecture}

With Proposition~\ref{prop:start}, Conjecture~\ref{conj-open} implies the following (recall that the inverse mapping is an involution):

\begin{conjecture} 
	Let $n \geq 5$. Every function $F \colon \F_{2^n} \rightarrow \F_{2^n}$ that is CCZ equivalent to the inverse function is already EA equivalent to it. Moreover, if $F$ is additionally a permutation then $F$ is affine equivalent to the inverse function.
\end{conjecture}

\subsection*{Acknowledgements}
{We would like to thank the anonymous referees for their careful reading of our paper and
their comments, which helped us to improve its presentation. We especially thank a referee who suggested 
an improvement in Theorem~\ref{thm:kloozeros} for the case $n$ even and provided background information on Kloosterman sums that helped us to improve the tutorial value of our paper. Remark \ref{rem:bound} is based 
on  comments from her/his report.
We thank Petr Lisonek for interesting discussions on Kloosterman zeroes and sending us the reference \cite{mth-pl}, which we used to compute Table~\ref{t:zeros}.}

This work was supported by the GAČR Grant 18-19087S -301-13/201843
\\

\bibliographystyle{acm}

\bibliography{waifi}
\end{document}